\documentclass[10pt,a4 paper]{article}
\usepackage{amssymb}
\usepackage{amsmath}
\usepackage{amsthm}
\usepackage{amsfonts}
\usepackage{amstext}
\usepackage{cite}
\usepackage{color}

\newtheorem{Theorem}{Theorem}
\newtheorem{Lemma}{Lemma}
\newtheorem{Proposition}{Proposition}

\newtheorem{Remark}{Remark}

\def\o{\Omega}

\newcommand{\real}{\mathbb{R}}
\newcommand{\tq}{\tilde{q}}
\newcommand{\tu}{\tilde{u}}

\newcommand{\tv}{\tilde{v}}

\newcommand{\calF}{\mathcal{F}}
\newcommand{\dif}{\,\mathrm{d}}
\newcommand{\di}{\mathrm{div}\,}
\newcommand{\p}{\partial}
\newcommand{\ds}{\displaystyle}

\def\C{\mathrm{curl}\,}
\def\D{\mathrm{div}\,}
\begin{document}
\title{ Convergence of the 2D Euler-$\alpha$ to Euler equations in the Dirichlet case: indifference to boundary layers\footnote{Email: mlopes@im.ufrj.br (M.C. Lopes Filho), hlopes@im.ufrj.br (H.J. Nussenzveig Lopes), etiti@math.uci.edu(E.S. Titi), }}
\author{  Milton C. Lopes Filho$^{\lowercase{a}}$,  Helena J. Nussenzveig Lopes$^{\lowercase{a}}$, \\ Edriss S. Titi$^{b,c}$,  Aibin Zang$^{\lowercase{a,d}\thanks{Corresponding author: Email: zangab05@126.com (A.B. Zang). }}$}
\date{March 20, 2014}
\maketitle

\begin{center}

$^a$ Instituto de Matem\'atica \; -- \;
Universidade Federal do Rio de Janeiro, \\
Cidade Universit\'aria -- Ilha do Fund\~ao, \;\;
Caixa Postal 68530,\\
21941-909 Rio de Janeiro, RJ -- Brasil.\\
$b$ Department of Computer Science and Applied Mathematics\\
Weizmann Institute of Science,
Rehovot, 76100, Israel.\\
$c$ Department of Mathematics and Department
of Mechanical and Aerospace Engineering \\
University of California,
Irvine, California 92697, USA.\\

 $^d$  Department of Mathematics, Yichun University, Yichun, Jiangxi, 336000, P.R.China\\

\end{center}

\begin{abstract}

In this article we consider the Euler-$\alpha$ system as a regularization of the incompressible Euler equations in a smooth, two-dimensional,
bounded domain. For the limiting Euler system we consider the usual non-penetration boundary condition, while, for the Euler-$\alpha$  regularization, we use velocity vanishing at the boundary. We also assume that the initial velocities for the Euler-$\alpha$ system approximate, in a suitable sense, as the regularization parameter $\alpha \to 0$, the initial velocity for the limiting Euler system. For small values of $\alpha$, this situation leads to a boundary layer, which is the main concern of this work. Our main result is that, under appropriate regularity assumptions, and despite the presence of this boundary layer, the solutions of the Euler-$\alpha$ system converge, as $\alpha \to 0$, to the corresponding solution of the Euler equations, in $L^2$ in space, uniformly in time.  We also present an example involving parallel flows, in order to illustrate the indifference to the boundary layer of the $\alpha \to 0$ limit, which underlies our work.


\end{abstract}
\numberwithin{equation}{section}

\numberwithin{equation}{section}

%

\section{Introduction}

Let $\Omega \subset \real^2$ be a bounded, simply connected domain, with smooth boundary $\partial \Omega$.
We denote by $\widehat{n}$ the exterior normal vector to $\partial \Omega$. We consider the initial-boundary-value problem for the Euler-$\alpha$ system in $\Omega$, with initial data $u_0^{\alpha} \in H^3(\o)$,  given by:

\begin{equation} \label{2dalphaeuler}
\left\{ \begin{array}{cl} \partial_t v^\alpha + u^\alpha \cdot \nabla v^\alpha + \displaystyle{\sum_{j=1}^2 } v^\alpha_j \nabla u^\alpha_j = - \nabla p^\alpha & \mbox{ in }~~ \o \times (0,\infty),\\[3mm]
\D u^\alpha = 0 & \mbox{ in }~~\o\times [0,\infty),\\[3mm]
u^\alpha = 0 \mbox{ on } \partial \Omega & \mbox{ on }~~\partial\o \times [0,\infty),\\[3mm]
u^\alpha(x,0) = u_0^{\alpha} & \mbox{ in } \o,
 \end{array} \right.
\end{equation}
where $v^\alpha = u^\alpha - \alpha^2 \Delta u^\alpha$.

Existence and uniqueness of a solution for  problem \eqref{2dalphaeuler} was established, by using geometric tools,   in \cite{MRS2000,S2000} for initial data $u_0^{\alpha} \in H^s(\Omega)$, $s>2$. Moreover, it is also remarked in the end of section 1 of  \cite{bu} that the global regularity of the two-dimensional Euler-$\alpha$ system \eqref{2dalphaeuler} follows clearly from \cite{ce}. Specifically, it is observed that, for fixed $\alpha$, the solutions of the viscous second-grade fluid established in \cite{ce} converges to the solution of \eqref{2dalphaeuler}, as the viscosity tends to zero. This in turn provides a direct traditional PDE proof for the global regularity of \eqref{2dalphaeuler}. We also remark that one can apply the abstract existence theorem of \cite{katolai1984} to \eqref{2dalphaeuler} in order to show the existence and uniqueness of solutions to \eqref{2dalphaeuler}.

Fix $u_0 \in H^3(\Omega)$, a divergence-free vector-field satisfying $u_0 \cdot \hat{n} = 0$ on $\partial\o$. We write the initial-boundary-value problem for the incompressible two-dimensional Euler equations in $\o$, with initial velocity $u_0$, as

\begin{equation} \label{2deuler}
\left\{ \begin{array}{cl} \partial_t \bar{u} + \bar{u} \cdot \nabla\bar{u} = - \nabla\bar{ p} & \mbox{ in }~~ \o \times (0,\infty),\\[3mm]
\D\bar{u} = 0 & \mbox{ in }~~ \o \times [0,\infty),\\[3mm]
\bar{u} \cdot \widehat{n} = 0 & \mbox{ on }~~ \partial \Omega \times [0,\infty),\\[3mm]
\bar{u}(x,0) = u_0 &\mbox{ in }~~ \o.
 \end{array} \right.
\end{equation}

Existence and uniqueness of a solution $\bar{u} \in C([0,\infty);(H^3(\Omega))^2)$ for \eqref{2deuler} can be found in \cite{katolai1984} (see also \cite{te1}) and references therein. Clearly, we also have $\bar{u} \in C^1([0,\infty);(H^2(\Omega))^2)$.

Next, we consider a family of initial data for the Euler-$\alpha$ system, $\{u_0^{\alpha}\} \subset H^3(\Omega)$, corresponding to $u_0$,  satisfying the following conditions:

\begin{equation} \label{E1}
\begin{array}{ll}
(\mathrm{i})  & u_0^{\alpha}~\mbox{ vanishes on}~ \partial \Omega,\\
(\mathrm{ii}) & u_0^\alpha \to u_0,~\mbox{ as} ~\alpha \to 0, \mbox{ in }~L^2(\Omega),\\
(\mathrm{iii}) & \|\nabla u_0^{\alpha}\|_{L^2}  = o(\alpha^{-1}),~ \mbox{ as } ~\alpha \to 0,~\mbox{ and}\\
(\mathrm{iv}) & \|u_0^{\alpha}\|_{H^3} = \mathcal{O}(\alpha^{-3}),~ \mbox{ as } ~\alpha \to 0.
\end{array}
\end{equation}

We call a family $\{u_0^{\alpha}\}$ satisfying \eqref{E1}  a {\em suitable family of approximations} to $u_0$.

Fix $T>0$ and let $u^{\alpha} \in C([0,T],(H^3(\o))^2\cap V) $ be the unique solution of \eqref{2dalphaeuler} with initial velocity $u_0^{\alpha}$, established,  e.g., by Theorem 2 in \cite{S2000} (see also earlier remarks concerning  \cite{bu} and \cite{ce}, and \cite{katolai1984}). In section 4, we present and prove the main result of the present article, namely that if $u^\alpha$ denotes the solution to \eqref{2dalphaeuler} with initial data $u_0^\alpha$ satisfying \eqref{E1}, then the sequence $\{u^{\alpha}\}$ converges, in $C([0,T]; (L^2(\Omega))^2)$, to the solution of \eqref{2deuler} with initial velocity $u_0$.
In \cite{kato} T. Kato introduced a criterion for the convergence of solutions of the incompressible Navier-Stokes equations, with no-slip boundary conditions, at the limit of vanishing viscosity, to solutions of the incompressible Euler equations with non-penetration boundary conditions. The proof of our main result, Theorem \ref{mainth}, borrows some ideas from \cite{kato}.

Using the eigenfunctions of the Stokes operator in the domain $\Omega$, we prove in section 5 that, for a given $u_0 \in (H^3(\Omega))^2$, with $\mbox{div }u_0=0$ and  $u_0 \cdot \hat{n} = 0$ on $\partial \Omega$, there exists a suitable family of approximations $\{u_0^{\alpha}\}$ to $u_0$.
Thus, using Theorem \ref{mainth}, we have that any smooth enough solution of Euler equations \eqref{2deuler}, with initial data $u_0$, can be approximated by a solution of \eqref{2dalphaeuler} in the $C([0,T];L^2(\o))-$norm.

In section 5 we also present  an example which illustrates the possible boundary layer behaviors of the $\alpha \to 0$ limit.

The Euler-$\alpha$ system \eqref{2dalphaeuler} was introduced as an {\it ad hoc} regularization of the incompressible Euler system, see \cite{hmr,bt}, and was later shown to have deep geometrical significance, as the Euler-Lagrange equations for geodesics on the group of volume-preserving diffeomorphisms with the right-invariant metric inherited from $H^1$, see \cite{MRS2000,S2000}. In addition, the Euler-$\alpha$ system corresponds to setting viscosity to zero in the second-grade fluid equations, which is a well-known non-Newtonian fluid model, see \cite{df}. Moreover, in the three-dimensional case the Euler-$\alpha$ system inspired the introduction of the Navier-Stokes-$\alpha$ and Leray-$\alpha$ viscous models, which turned out to be remarkable sub-grid scale models of turbulence (see, e.g., \cite{Cao,cfh1,cfh2,cfh3,cheskidov,fht1,fht2}, and references therein).

There has been substantial work on the Euler-$\alpha$ system. In the full plane, well-posedness has been studied under different regularity assumptions, see \cite{blt1,blt,bu,OS2001}. Also, the vanishing viscosity limit of second-grade fluids to Euler-$\alpha$ was established in \cite{bu} and the limit $\alpha \to 0$ of Euler-$\alpha$ to Euler was investigated in \cite{blt1,blt, lt}. In domains with boundary, besides the non-penetration condition $u\cdot \widehat{n}=0$, the Euler-$\alpha$ system requires additional boundary conditions, but there is no natural choice for them, either on physical or geometric grounds. There are two different kinds of boundary conditions considered in the literature: Navier-type slip conditions and homogeneous  Dirichlet  boundary conditions (no-slip). Existence and uniqueness of solutions to the Euler-$\alpha$ system in a bounded domain, under Navier conditions was established in \cite{BR2003, S2000}. The limit as $\alpha \to 0$ of second-grade fluids to the Navier-Stokes equations was studied, for flow in a bounded domain with Dirichlet boundary conditions, in \cite{iftimie, ct}. As mentioned earlier, in \cite{bu}, it was remarked that the uniform estimates, with respect to the viscosity, that have been established in \cite{ce} will easily imply the convergence of the solutions of the second-grade fluid equations, as the viscosity $\nu\to 0$, and fixed $\alpha$, to the corresponding unique solutions of the Euler-$\alpha$ equation under homogeneous  Dirichlet  boundary conditions.  In \cite{twocouples}, the independent limits of second-grade fluids, as $\alpha \to 0$ or $\nu \to 0$, were studied for flows in a bounded domain with Navier-type boundary conditions. In all singular limits studied, in the presence of boundaries, the difficulty of dealing with a boundary layer was avoided. The main purpose of the present work is  to address precisely this difficulty.

The $\alpha$--regularization, under homogeneous Dirichlet boundary conditions, as considered here, has two advantages: (a) it is particularly simple and (b) it formally resembles the effect of viscosity. However, our analysis ends up highlighting the sharp contrast between small viscosity, in the context of the Navier-Stokes equations, and small $\alpha$, in the context of the Euler-$\alpha$ equations, in the presence of rigid boundaries.
The initial objective of the present investigation was to obtain a version of the Kato criterion in the vanishing $\alpha$ limit. The convergence which we obtained here was unexpected, and it certainly appears in other contexts, such as the three-dimensional case, combining small $\alpha$ and small viscosity in case of second-grade fluid (cf. \cite{LNTZ}), or by considering other $\alpha$-type regularizations of the ideal flow equations. We chose to focus, in this article, on the simplest case in order to provide an accessible baseline for future research.

The remainder of this paper is organized as follows. In section 2, we will introduce notation, present some preliminary results and write the vorticity formulation of \eqref{2dalphaeuler}. In section 3, we include a proof of global existence and uniqueness of a weak solution for \eqref{2dalphaeuler}. Although this result can be found explicitly in \cite{MRS2000,S2000} (or indirectly in \cite{katolai1984}, or in \cite{bu} combined with \cite{ce}),  we require, for our main result, some explicit estimates, which are derived in the proof of Theorem \ref{wellposed}. In section 4, we obtain, for any $T \in (0,\infty)$, the convergence of solutions of the Euler-$\alpha$ equations to  solutions of the Euler equations, as $\alpha\to 0$, in $C([0,T]; (L^2(\o))^2)$, assuming that the initial data for the Euler-$\alpha$ system is a suitable family of approximation to the initial data for the Euler equations. In section 5, we describe a method for constructing a suitable family of approximations for a given initial velocity $u_0$ of Euler equation \eqref{2deuler}.  We also present a class of examples illustrating the boundary layer behavior of the small
$\alpha$ approximation and we discuss some directions for future research.

\section{Notations and preliminaries}
In this section, we introduce notation and we present the vorticity formulation of the Euler-$\alpha$ system.

We use the notation $H^{m}(\o)$ for the usual $L^2$-based Sobolev spaces of order $m$, with the norm $\|\cdot\|_m$ and the scalar product
$(\cdot,\cdot)_m$. For the case $m=0$, $H^0(\o)=L^2(\o)$; we denote the corresponding norm by $\|\cdot\|$ and the inner product by $(\cdot,\cdot)$.
We denote by $C^{\infty}_c(\o)$ the space of smooth functions, compactly supported in $\o$, and  by $H^m_0(\o)$ the closure of $C^{\infty}_c(\o)$ under the $H^m$-norm.

We also make use of the following notation:
\begin{equation*}
\begin{aligned}
&\left[u,v\right]=(\nabla u,\nabla v)+\alpha^2(\Delta u,\Delta v), ~\mbox{for all}~ u,v\in C_c^\infty(\Omega),\\
&H=\{u\in (L^2(\Omega))^2: \D u=0~\mbox{in}~\Omega,~ u\cdot \hat{n}=0~\mbox{on}~\p\Omega\},\\
&V=\{u\in (H^1_0(\Omega))^2: \D u=0~\mbox{in}~\Omega\},\\
&\dot{H}^1=\{\pi\in H^1(\o): \int_\Omega \pi\, \dif x=0\}.
\end{aligned}
\end{equation*}
It is easy to see that, for each fixed $\alpha>0$, $[\cdot,\cdot]$ gives rise to an inner product on $H^2_0(\o)$ and that the
corresponding norm is equivalent to the usual $H^2$-norm, restricted to $H^2_0(\Omega)$.

Let $u =(u_1,u_2) \in V$. Then
\[\C u \equiv \partial_{x_1} u_2 - \partial_{x_2} u_1 = \nabla^{\perp} \cdot u,\] where
$\nabla^{\perp} = (-\partial_{x_2},\partial_{x_1})$.

Hereafter we use $C$ for constants that, in principle, depend on $\alpha$, and $K$ for those that are independent of $\alpha$.

The following results can be found, for example, in \cite{Galdi}.

\begin{Lemma} \label{InvertCurl}
Let $\phi\in  H^m(\o), m\ge0$. Then there exists a unique divergence-free vector field $\Psi\in (H^{m+1}(\o))^2$ with $\Psi\cdot \hat{n}=0$ on $\partial \o$,
such that
\begin{equation*}
\begin{aligned}
&\C \Psi=\phi,\\
&\|\Psi\|_{m+1}\le K\|\phi\|_{m},
\end{aligned}
\end{equation*}
for some constant $K>0$, which depends only on $m$ and $\Omega$.
\end{Lemma}

Next we introduce the potential vorticity. Given $u=(u^1,u^2)$ a solution of the Euler-$\alpha$ system \eqref{2dalphaeuler},  the associated potential vorticity $q$ is defined by
$$q \equiv \C (u - \alpha^2 \Delta u) = \partial_{x_1} (u^2 - \alpha^2 \Delta u^2) - \partial_{x_2}(u^1 - \alpha^2\Delta u^1).$$

We apply the curl operator to the first equation in \eqref{2dalphaeuler}  and, after a straightforward calculation, we obtain the vorticity
formulation of the Euler-$\alpha$ equations:
\begin{equation}\label{eq8}
\left\{\begin{array}{cl} \partial_t q + u\cdot\nabla q = 0 , &\mbox{in}~~\Omega\times (0, \infty) ,\\[2mm]
\D u= 0, &\mbox{in}~~\Omega\times [0, \infty), \\[2mm]
\C (u - \alpha^2 \Delta u) = q, &\mbox{in}~~\o\times[0,\infty)\\[2mm]
u=0, &\mbox{on}~~\partial \o \times [0,\infty)\\[2mm]
q(\cdot,0)=\C(u_0-\alpha^2\Delta u_0)=q_0&\mbox{in}~~\Omega.
\end{array}\right.
\end{equation}

Assume $(u,q)$ is a solution of \eqref{eq8}. Let us introduce the stream function $\phi$, such that $u=\nabla^{\perp}\phi=(-\phi_{x_2},\phi_{x_1})$.
After appropriately fixing an additive constant, it is easy to see that $\phi$ satisfies the elliptic problem:
 \begin{equation}\label{eq10}
\left\{\begin{array}{cl}\Delta \phi-\alpha^2\Delta^2\phi=q, &\mbox{in}~~\Omega\\[3mm]
\phi=\frac{\p\phi}{\p\hat{n}}=0, &\mbox{on}~~\p\Omega.
\end{array}\right.
\end{equation}

\begin{Lemma} \label{ExistStream}
Let $q\in L^2(\Omega)$.  There exists a unique solution $\phi \in H_0^2(\Omega)$ of (\ref{eq10}), in the following sense:
\begin{equation}\label{eq11}
\left[\phi,\psi\right]= (-q,\psi), ~\mbox{for any }\psi\in H_0^2(\Omega).
\end{equation}
Furthermore, the solution operator $q \mapsto \phi$ maps $L^2(\o)$ continuously into  $H^4(\o) \cap H^2_0(\o)$.
 \end{Lemma}

\begin{proof}
We define the bilinear operator $\mathcal{A}(\phi,\psi)=[\phi,\psi]$, for $\phi,\psi\in H^2_0(\o)$. It is easy to see that
$$|\mathcal{A}(\phi,\psi)|\le C\|\phi\|_2\|\psi\|_2$$
and, also, that
$$\mathcal{A}(\phi,\phi)=[\phi,\phi]=(\nabla \phi,\nabla \phi)+\alpha^2(\Delta \phi,\Delta \phi)\ge C\|\phi\|_2^2,$$  where $C>0$ depends only on $\alpha$ and $\o$.
Using the Lax-Milgram theorem (cf. \cite{cf, Ev}), we  obtain existence and uniqueness of $\phi \in H^2_0(\Omega)$ satisfying \eqref{eq11}.

Next, we will show that the solution operator $q \mapsto \phi$ is continuous from $L^2(\o)$ into  $H^4(\o) \cap H^2_0(\o)$. Indeed, from  Lemma \ref{InvertCurl}, there exists a unique divergence-free vector field $\Phi\in (H^1(\Omega))^2$, with $\Phi\cdot \hat{n}=0$ on $\p\Omega$,  such that
\begin{equation} \label{PhiH1}
\C \Phi=  q~\mbox{ and }~\|\Phi\|_1\le K\|q\|.
\end{equation}

 It is easy to see from \eqref{eq11} that $\phi$ satisfies $\Delta\phi-\alpha^2\Delta^2\phi=q$ in $D'(\Omega).$  Hence we have, in the sense of distributions, the identity
  $$\C (-\alpha^2\Delta(\nabla^{\perp}\phi)- (\Phi-\nabla^{\perp} \phi))=0.$$
 Therefore, since $\Omega$ was assumed to be simply connected, there exists a unique pressure $\pi\in\dot{H}^1$, associated with the irrotational vector field $-\Delta(\nabla^{\perp}\phi)- \frac{1}{\alpha^2}(\Phi-\nabla^{\perp} \phi),$ so that
  \begin{equation}\label{eq12}
\left\{\begin{array}{cl}-\Delta (\nabla^{\perp}\phi)+\nabla \pi=f, &\mbox{in}~~\Omega\\[2mm]
 \D (\nabla^{\perp}\phi)= 0 &\mbox{in}~~\Omega\\[2mm]
\nabla^{\perp}\phi=0 &\mbox{on}~~\p\Omega,
\end{array}\right.
\end{equation}
  where $f=\ds{\frac{1}{\alpha^2}}(\Phi-\nabla^{\perp} \phi) \in H^1(\Omega)$.  From standard estimates on the Stokes operator (see, for example, Lemma IV.6.1 in\cite{ga2}), we have
\begin{equation}\label{est1}
\begin{aligned}
\|\nabla^{\perp}\phi\|_{3}\le K\|f\|_{1}\le \frac{K}{\alpha^2}(\|\Phi\|_{1}+\|\nabla^{\perp}\phi\|_{1}).
\end{aligned}
\end{equation}
Using (\ref{eq11}) with $\psi=\phi$, we obtain, thanks to the Poincar\'e inequality \cite{Ev}
$$\|\nabla \phi\|^2 + \alpha^2\|\Delta\phi\|^2 \le \|\phi\|\|q\| \leq \|\nabla \phi \| \|q\| \frac{1}{\lambda_1^{1/2}},$$
where $\lambda_1$ is the first eigenvalue of the Laplace operator on $\Omega$ with Dirichlet conditions.
Applying  Young's inequality we find that
$ \alpha^2\|\Delta\phi\|^2\le \frac{1}{2\lambda_1^{1/2}}\|q\|^2,$
which, in turn, implies that
\begin{equation} \label{gradphiH1}
\|\nabla^{\perp}\phi\|_{1}\le \frac{K}{\alpha \,\lambda_1^{1/2}}\|q\|,
\end{equation}
by standard elliptic regularity estimates together with the Poincar\'e inequality. Finally,  we use \eqref{PhiH1} and \eqref{gradphiH1} in \eqref{est1}, and we recall that we are interested in the small $\alpha$ regime, say
$\alpha \in (0,\lambda_1^{-1/2})$, we hence obtain
\begin{equation} \label{gradphiH3}
\|\nabla^{\perp}\phi \|_3 \le \frac{K}{\alpha^3}\|q\|.
\end{equation}
It follows from this estimate, together with the Poincar\'e inequality, that $\phi\in H^4(\Omega),$ and that $\|\phi\|_4 \leq \displaystyle{\frac{K}{\alpha^3}\|q\|}$.
\end{proof}
\begin{Remark} \label{BSalpha}
In view of Lemma \ref{ExistStream}, we can now introduce the bounded linear operator $\mathbb{K}: L^2(\o) \to H^3(\o) \cap W^{1,\infty}_0(\o)$, given by $q \mapsto u=\mathbb{K}[q] = \nabla^{\perp}\phi$, where $\phi$ is the unique solution of \eqref{eq10}. We will refer to $\mathbb{K}$ as the Biot-Savart-$\alpha$ operator.
\end{Remark}

\section{Global well-posedness of Euler-$\alpha$ equation}
In this section we will establish global-in-time existence and uniqueness  of a weak solution to the Euler-$\alpha$ equations
\eqref{2dalphaeuler}, see Theorem \ref{wellposed} below.

Recall the Biot-Savart-$\alpha$ operator $\mathbb{K}$ introduced in Remark \ref{BSalpha}.

\begin{Theorem} \label{wellposed}
Fix $T>0$. Let $q_0\in L^2(\o)$, and set $u_0=\mathbb{K}(q_0)$. Then there exists a unique function $q\in C([0,T];L^2(\Omega)) $ and a unique vector field $u=\mathbb{K}(q)\in C([0,T];(H^3(\o))^2 \cap V)$, such that the pair $(u,q)$ is a weak solution of \eqref{eq8} in the following sense:
\\ For any test function $v\in C_0^\infty(\Omega)$  it holds that
\begin{equation}\label{eq13}
\begin{aligned}
( q(t),v)_{L^2}-( q_0,v)_{L^2}-\int_{0}^{t}\int_\Omega( u\cdot\nabla v )q  \dif x\dif t=0,
\end{aligned}
\end{equation}
for every $t\in [0, T]$. Moreover,
\begin{equation}\label{qbd}
\|q(t)\| \leq \|q_0\|,~\mbox{ for all}~t\in [0,T].
\end{equation}
\end{Theorem}

\begin{Remark} \label{OliverEtAlli}
In \cite{ko}, the authors  consider the Euler-$\alpha$ equation with Navier (slip) boundary conditions, and they prove the existence of solution by constructing the solution as the   limit of viscous regularization of the $\alpha-$model.  Here, we  will use the Banach fixed point theorem.
\end{Remark}

\begin{proof}
We begin by constructing a mapping $\mathcal{F}$ from $C([0,T]; V)$ to itself which, subsequently, we will show is a contraction.
For simplicity's sake we first consider the vorticity formulation of the Euler-$\alpha$ equations.

Let $u \in C([0,T]; V)$. It follows from the existence, uniqueness and regularity results of the DiPerna-Lions  \cite{dl}, that the following {\em linear} problem has a unique weak (distributional) solution $\tilde{q} \in C ([0,T];L^2(\o))$:

\begin{equation} \label{qMapping}
\left\{
\begin{array}{l}
\partial_t \tq + u \cdot \nabla \tq = 0, \\
\tq(0,\cdot) = q_0.
\end{array}
\right.
\end{equation}
Moreover, the following estimate holds true:

\begin{equation} \label{tqEst}
 \|\tq (t\| \leq \|q_0\|, ~\mbox{ for all}~t\in [0,T].
\end{equation}
Next, we introduce a new velocity, $\tu$,  constructed as follows:
\[\tu = \nabla^{\perp} \tilde{\phi}, \text{ where } \tilde{\phi} \in H^4(\o) \cap H^2_0(\o), \text{ and }\]
\[\Delta \tilde{\phi} - \alpha^2 \Delta^2 \tilde{\phi} = \tq, \text{ in } [0,T] \times \o.\]
It follows that $\tu =  \mathbb{K}[\tq]$. In view of Lemma \ref{ExistStream} and Remark \ref{BSalpha}, it follows that
$\tu \in C([0,T];(H^3(\o))^2 \cap V)$.

We introduce the mapping $\calF:  C([0,T];V\cap (H^3(\o))^2)  \to  C([0,T];V\cap (H^3(\o))^2) $ as
\[u \mapsto \calF[u] := \tu.\]
We easily obtain that
\[ \sup_{t\in [0,T]}\|\calF[u](t)\|_1 \leq C\sup_{t\in [0,T]}\|\tq (t)\| \leq C\|q_0\|.\]
In fact, in view of \eqref{gradphiH3}, as established in Lemma \ref{ExistStream}, we have even more:

\begin{align} \label{evenmore}
 \sup_{t\in [0,T]}\|\calF[u](t) \|_3 \leq C\sup_{t\in [0,T]}\|\tq(t) \| \leq C\|q_0\|.
\end{align}
Let $\tv := \tu - \alpha^2 \Delta \tu$.
Next, we note that $(\tu,\tv)$ is a solution of the following modified Euler-$\alpha$ system:

\begin{equation} \label{modEuleralpha}
\left\{
\begin{array}{ll}
\partial_t \tv + u \cdot \nabla \tv - \sum_j u_j \nabla \tv_j + \nabla p = 0, & \text{ in } (0,T) \times \o , \\
\D \tu = 0, & \text{ in } [0,T]\times\o , \\
\tu = 0, & \text{ on } [0,T]\times\partial\o , \\
\tu(0,\cdot) = u_0, & \text{ on } \o .
\end{array}
\right.
\end{equation}
Indeed, one has the identity $$\partial_t \tq + u \cdot \nabla \tq =\C(\partial_t \tv + u \cdot \nabla \tv - \sum_j u_j \nabla \tv_j).$$ Thanks to \eqref{qMapping}, one concludes that
$$\C(\partial_t \tv + u \cdot \nabla \tv - \sum_j u_j \nabla \tv_j)=0.$$ Since $\Omega$ is simply connected, there exists a pressure $p$ such that
$$\partial_t \tv + u \cdot \nabla \tv - \sum_j u_j \nabla \tv_j =- \nabla p.$$
Thus the first equation of \eqref{modEuleralpha} holds. We use system \eqref{modEuleralpha} to show that, for some sufficiently small $\delta > 0$, $\calF$ is a contraction with respect to the norm  $C([0,\delta];V)$. To this end let $u^1$ and $u^2$ be divergence-free vector fields in $C([0,\delta];V \cap (H^3(\o))^2)$, for some $\delta > 0$ to be fixed later. Consider $\tu^1$, $\tu^2$, $\tv^1 = \tu^1 - \alpha^2\Delta\tu^1$ and $\tv^2=\tu^2 - \alpha^2\Delta\tu^2$. Set
\[R = u^1 - u^2,\]
\[S = \tu^1 - \tu^2 \equiv \calF[u^1] - \calF[u^2].\]
Note that
\[\tv^1 - \tv^2 = S - \alpha^2 \Delta S.\]
Subtracting the equation for $\tu^2$ from that for $\tu^1$ we obtain:
\begin{equation} \label{diff}
\partial_t (S-\alpha^2\Delta S) + u^1 \cdot \nabla \tv^1 - u^2 \cdot \nabla \tv^2 - \sum_j u^1_j \nabla \tv^1_j + \sum_j u^2_j\nabla\tv^2_j + \nabla p^1 - \nabla p^2 = 0.
\end{equation}
Take the scalar product of \eqref{diff} with $S$, re-write the nonlinear terms using $R$ and $S$ and integrate over $\o$ to obtain:
\begin{align} \label{differentialest1}
&\frac{1}{2}\frac{\dif}{\dif t} (\|S\|^2 + \alpha^2\|\nabla S\|^2)  \\
&= \;\;\; - \int_{\o} S \cdot [(R \cdot \nabla )\tv^1 + (u^2\cdot\nabla) (S-\alpha^2\Delta S)] \dif x \nonumber \\
& \;\;\;\;\;\;\; + \int_{\o} S \cdot \left[\sum_j u^1_j \nabla (S-\alpha^2\Delta S)_j + \sum_j R_j\nabla\tv^2_j\right]\dif x  \nonumber \\
& =: I + J. \nonumber
\end{align}
We begin by estimating first   $I$. We note, as usual, that $(S,u^2\cdot \nabla S) = 0$, so that we find:
\begin{align*}
&|I| \leq \left|\int_{\o} S \cdot [(R \cdot \nabla )\tv^1 - (u^2\cdot\nabla) \alpha^2\Delta S)] \dif x \right| \\
&= \left|\int_{\o} S \cdot [(R \cdot \nabla )\tv^1] + [(u^2\cdot\nabla) S] \cdot\alpha^2\Delta S) \dif x \right|\\
&= \left|\int_{\o} S \cdot [(R \cdot \nabla )\tv^1] + \alpha^2\sum_{k,\ell} [\partial_{\ell} u^2_k\partial_{k}S + u^2_k\partial_k\partial_{\ell}S]\partial_{\ell} S) \dif x \right|\\
&= \left|\int_{\o} S \cdot [(R \cdot \nabla )\tv^1] + \alpha^2\sum_{k,\ell} [\partial_{\ell} u^2_k\partial_{k}S]\partial_{\ell} S) \dif x \right|,
\end{align*}
where we integrated by parts the term with the Laplacian and then used the divergence-free condition on $u^2$ to show that the remaining term with two derivatives of $S$ vanishes.
Therefore, using H\"older's inequality, we deduce that
\[|I|\leq \|S\|_{L^4}\|R\|_{L^4}\|\nabla\tv^1\| + \alpha^2 \|\nabla u^2\|_{L^{\infty}} \| \nabla S\|^2, \]
so that, using the Sobolev inequality, we get
\begin{align} \label{Aest}
|I|\leq C \|\nabla S\|\|\nabla R\|+ C\alpha^2   \|\nabla S\|^2.
\end{align}

Next, we estimate the second integral term, $J$. We find, using H\"older's inequality together with the divergence-free condition on $S$, that:
\begin{align*}
&|J| \leq \left|\int_{\o} S \cdot \left[\sum_j u^1_j \nabla (S-\alpha^2\Delta S)_j + \sum_j R_j\nabla\tv^2_j\right]\dif x \right|\\
&\leq \left| \int_{\o} \sum_j u^1_j S \cdot \nabla S_j   \dif x - \alpha^2 \int_{\o} \sum_j u^1_j \D(\Delta S_j\,S) \dif x \right| +
\|S\|_{L^4} \|R\|_{L^4} \|\nabla \tv^2\|  \\
&\leq \|u^1\|_{L^{\infty}} \|S\| \|\nabla S\| + \alpha^2 \left| \int_{\o} \sum_j \nabla u^1_j \cdot S\, \Delta S_j  \dif x \right| + \|S\|_{L^4} \|R\|_{L^4} \|\nabla \tv^2\|  \\
&=\|u^1\|_{L^{\infty}} \|S\| \|\nabla S\| + \alpha^2 \left| \int_{\o} \sum_{j,k} \nabla u^1_j \cdot S \partial_k\partial_k \, S_j  \dif x \right| + \|S\|_{L^4} \|R\|_{L^4} \|\nabla \tv^2\|\\
&\leq \|u^1\|_{L^{\infty}} \|S\| \|\nabla S\| + \alpha^2 \left| \int_{\o} \sum_{j,k} \partial_k\nabla u^1_j \cdot S \, \partial_k \, S_j  +
\nabla u^1_j \cdot \partial_k S  \, \partial_k \, S_j \dif x \right| + \|S\|_{L^4} \|R\|_{L^4} \|\nabla \tv^2\|  \\
&\leq \|u^1\|_{L^{\infty}} \|S\| \|\nabla S\| + \alpha^2 \sum_{j,\ell}\|\partial_k\partial_{\ell} u^1_j\|_{L^4}\|S\|_{L^4}\|\nabla S\|
+ \alpha^2 \|\nabla u^1\|_{L^{\infty}} \|\nabla S \|^2
 + \|S\|_{L^4} \|R\|_{L^4} \|\nabla \tv^2\| ,\end{align*}
where we integrated by parts the term with the Laplacian.
Therefore, using the Sobolev inequality, followed by Young's inequality, together with the uniform bound \eqref{evenmore}, we arrive at
\begin{align} \label{Best}
|J| \leq C\|S\|\|\nabla S\| + C\alpha^2 \|\nabla S \|^2 + C \|\nabla S \| \|\nabla R\|.
\end{align}
Insert the estimates derived in \eqref{Aest} and \eqref{Best} into \eqref{differentialest1} leads to the differential inequality

\begin{align} \label{differentialest2}
\frac{1}{2}\frac{\dif}{\dif t} (\|S\|^2 + \alpha^2\|\nabla S\|^2)  \\
& \leq  C \|\nabla S\|\|\nabla R\|+ C\alpha^2   \|\nabla S\|^2 + C\|S\|\|\nabla S\| \nonumber \\
& \leq C_1 (\|S\|^2 + \alpha^2\|\nabla S\|^2) + C_2(\|R\|^2 + \alpha^2\|\nabla R\|^2).\nonumber
\end{align}
Recall that $S(t=0)=0$, since $u^1(t=0) = \tu^1(t=0) = \tu^2 (t=0)= u^2 (t=0) = u_0$. Hence, we obtain by Gronwall's inequality, that

\begin{align}        \label{almostthere}
(\|S\|^2 + \alpha^2\|\nabla S\|^2)(t) \leq \int_0^t (\|S\|^2 + \alpha^2\|\nabla S\|^2)(s) \, e^{C(t-s)} \dif s.
\end{align}
Taking the supremum, for $t \in [0,\delta]$, of the norms $(\|S\|^2 + \alpha^2\|\nabla S\|^2)(t)$ we deduce
\begin{align} \label{gronwall}
\sup_{t\in [0,\delta]} (\|S\|^2 + \alpha^2\|\nabla S\|^2)(t) \leq \frac{e^{C \delta} - 1}{C} \sup_{t \in [0,\delta]} (\|R\|^2 + \alpha^2\|\nabla R\|^2)(t).
\end{align}
Therefore, if we choose $\delta>0$ small enough, so that $\sigma=\ds{\frac{e^{C \delta} - 1}{C}} < 1$ then we have shown that $\calF$ is a contraction with respect to the $H^1$-norm, for short interval of time $[0,\delta]$.
Ineed, we have obtained the estimate
\begin{align} \label{contraction}
\sup_{t \in [0,\delta]} \|\calF[u^1]-\calF[u^2]\|_1(t) \leq \sigma  \sup_{t \in [0,\delta]} \|u^1-u^2\|_1(t).
\end{align}
We invoke the Banach fixed point theorem in metric spaces to conclude the existence of a unique fixed point $u \in C([0,\delta];V)$. This fixed point is also the limit of the fixed point iteration, where $u^0\equiv u_0$ and $u^n \equiv \calF[u^{n-1}]$, as the argument. We easily know that the sequence $\{u^n\}$ converge to $u$ in $C([0,T];V)$.  As $u_0 \in (H^3(\Omega))^2$ it follows from \eqref{evenmore} that
\[\sup_{t\in [0,\delta]}\|u^n(t)\|_3 \leq C\|q_0\|,\]
for all $n$. Hence, by the Banach-Aloaglu theorem there exists a subsequence $\{u^{n_k}\}$ which converges, weak-$\ast$ in $L^{\infty}((0,\delta);(H^3(\o))^2)$, to a limit in the same space. As this subsequence also converges strongly in $C([0,\delta];V)$ to the unique fixed point $u$, it follows, by uniqueness of limits, that the fixed point belongs to the more regular space $L^{\infty}((0,\delta);(H^3(\o))^2) $.

In fact, since $u\in C([0,\delta];V)$, we also have $q\in C([0,\delta];L^2(\o))$ from the uniqueness and regularity reuslts in \cite{dl} for the transport equation \eqref{qMapping}.  Clearly, $u$ is a solution of \eqref{eq13} with $q = \C v$, $v=u-\alpha^2\Delta u$.   Consequently, $u\in C([0,\delta];(H^3(\o))^2).$ Therefore, it follows that $u$ is a distributional solution of \eqref{2dalphaeuler}. Since the $C([0,\delta];(H^3(\Omega))^2)$-norm of $u$ is bounded independent of $\delta$, we can repeat the argument above  and extend the solution to any interval $[0,T].$

\end{proof}

 \section {Convergence as $\alpha\to 0$}

 In  \cite{lt}, the authors have studied the convergence  of smooth solutions of the Euler-$\alpha$ to corresponding solutions of the Euler equations, as $\alpha\to 0,$ in whole space. In this section, we will prove that  the solutions $\{u^\alpha\}$ of Euler-$\alpha$ equations, with Dirichlet boundary conditions,  converge to the unique solution $\bar{u}$ of Euler equations, as $\alpha\to 0$ . Specifically, we state and prove the following theorem which is the main result in this paper:

\begin{Theorem} \label{mainth}
Fix $T>0$, and let $u_0\in (H^3(\Omega))^2\cap H$.  Assume also that we are given
a suitable family of  approximations $\{u_0^{\alpha} \}_{\alpha > 0} \subset (H^3(\Omega))^2$ for $u_0,$ satisfying (\ref{E1}). Suppose that  $u^\alpha \in C([0,T];(H^3(\Omega))^2)$ is the unique  solution of Euler-$\alpha$  with initial velocity $u^\alpha_0$, established in Theorem \ref{wellposed}. Let $\bar{u} = \bar{u}(t,x) \in C([0,T];(H^3(\Omega))^2)\cap C^1([0,T];(H^2(\o))^2)$ be the unique strong solution of the incompressible Euler equations with initial velocity $u_0$.  Then
\begin{equation}\label{Rlimit}
\begin{aligned}
\lim_{\alpha\to 0}\sup_{t\in [0,T]}\|u^\alpha (t)-\bar{u}(t)\|=0, ~\mbox{and}~\lim_{\alpha\to 0}\sup_{t\in [0,T]}\alpha^2\|\nabla u^\alpha(t)\|=0.
\end{aligned}
\end{equation}
\end{Theorem}

In \cite{kato} T. Kato established a criterion for the convergence, of the vanishing viscosity limit of solutions of the Navier-Stokes equations subject to the homogeneous Dirichlet boundary conditions,  to a solution of the Euler equations in domains with physical boundaries. The proof of Theorem \ref{mainth} is inspired by Kato's argument. The main ingredient consists of establishing a boundary layer corrector function for the discrepancy between $u^\alpha$ and $\bar{u}$ near the boundary. To construct this boundary  corrector function we  consider, first, the stream function $\bar{\psi}=\bar{\psi}(t,x)$ associated to $\bar{u}$, given by the unique solution of the elliptic equation
\begin{equation} \label{barpsieq}
\left\{
\begin{aligned}
& \Delta \bar{\psi} = \C \bar{u}, & \text{ in } \Omega,\\
& \bar{\psi} = 0, & \text{ on } \partial \Omega.
\end{aligned}
\right.
\end{equation}
It follows classically that
\[\bar{u} = \nabla^{\perp}\bar{\psi}.\]

%
Let $\xi: \mathbb{R}^+\to [0,1]$ be a smooth cut-off function such that
\begin{equation}\label{22}
\xi(0)=1, \;\;\;\xi(r)=0~\mbox{for}~r\ge 1.
\end{equation}
Let $\delta>0,$ be small enough to be determined later, and set
\begin{equation}\label{23}
z=z(x)= \xi\left(\frac{\rho}{\delta}\right), \; ~\mbox{where}~\rho= \mathrm{dist} (x,\p\o), ~\mbox{for any } ~x\in\bar{\o}.
\end{equation}
We introduce the boundary layer corrector $u_b=u_b(t,x)$ as

\begin{equation}\label{24}
u_b= \nabla^{\perp}(z\bar{\psi}).
\end{equation}
We collect below some useful estimates on the boundary layer corrector function.

\begin{Lemma} \label{ubest}
Let $u_b$ be defined by (\ref{24}). Then we have that:
\begin{equation}\label{26}
\begin{aligned}
&\sup_{t\in [0,T]} \|\p_t^{\ell}u_b(t)\|\le K\delta^{\frac{1}{2}}, \sup_{t\in [0,T]} \|\p^{\ell}_t\nabla u_b(t)\|\le K\delta^{-\frac{1}{2}},\\
&\sup_{t\in [0,T]}\|\rho^2\nabla u_b(t)\|_{L^\infty}\le K\delta,  \sup_{t\in [0,T]}\|\rho\nabla u_b(t)\|\le K\delta^{\frac{1}{2}},
\end{aligned}
\end{equation}
where $\ell=0,1$ and $K$ depends only on $\bar{u}, \xi$ and $\o$, but does not depend on $\delta$.
\end{Lemma}
We observe that these estimates follow by straightforward calculations and we omit their proof (cf.\cite{kato}).

\vspace{0.3cm}

We are now ready to give the proof of our main result, Theorem \ref{mainth}.

\begin{proof}[Proof of Theorem \ref{mainth}.]
We start with the observation that, since $u^\alpha \in C([0,T];V\cap (H^3(\o))^2)$. We multiply the Euler-$\alpha$ equations \eqref{2dalphaeuler} by $u^\alpha$ and integrating over time and space, and use the hypotheses \eqref{E1}, we obtain that
\begin{equation}\label{enest}
\|u^{\alpha}\|^2 + \alpha^2 \|\nabla u^{\alpha}\|^2 =
\|u^{\alpha}_0\|^2 + \alpha^2 \|\nabla u^{\alpha}_0\|^2 \leq K.
\end{equation}
%
Since $\D u^\alpha =0$, we have  from \eqref{enest}
\begin{equation} \label{H1estualpha}
\|\C u^\alpha\|= \|\nabla u^\alpha\| \le \frac{K}{\alpha}.
\end{equation}
Recall that $q^{\alpha} = \C (u^\alpha - \alpha^2 \Delta u^\alpha)$, then by theorem 1, \eqref{H1estualpha} and \eqref{E1}, We have, for all $t\in [0,T]$,
\begin{equation*}
\begin{aligned}
 \|q^{\alpha}(t)\| &\le \|q^{\alpha}_0\| \leq \| \C u^\alpha_0\| + \alpha^2 \|u^\alpha_0\|_{H^3}\le\frac{K}{\alpha}.
\end{aligned}
\end{equation*}
by our assumptions \eqref{E1}.
From the above and \eqref{H1estualpha} we have
\begin{equation*}
\begin{aligned}
\alpha^2\|\Delta\,\C u^\alpha\|\le\|q^\alpha\|+\|\C u^\alpha\|\le  \frac{K}{\alpha}.
\end{aligned}
\end{equation*}
Finally, we conclude that, for all $t\in [0,T]$,
 \begin{equation}\label{281}
 \|u^\alpha (t)\|_3\le \frac{K}{\alpha^3},
 \end{equation}
where $K$ is independent of $\alpha$.

Set $W^\alpha=u^\alpha-\bar{u}$, then from \eqref{2dalphaeuler}  and \eqref{2deuler}, $W^\alpha$ satisfies
\begin{equation}\label{28}
\left\{\begin{array}{cl} \partial_t W^\alpha +  ( u^\alpha \cdot\nabla) W^\alpha+(W^\alpha\cdot\nabla)\bar{u}
= & \\ \\
\;\;\;\;\;\;\;\;\;\;\;\;-\ds{\nabla\left(p^\alpha - \bar{p} +\frac{|u^\alpha|^2}{2}\right) + \D \sigma^\alpha}, &\mbox{in}~~\Omega\times (0, T) ,\\ \\ 
\D W^\alpha= 0,&\mbox{in}~~\Omega\times (0, T), \\ \\
W^\alpha\cdot\vec {n}=0,  &\mbox{on}~~\p\Omega\times (0, T), \\ \\
W^\alpha(0,x)=u^\alpha_0-u_0, & \mbox{in}~~\Omega,
\end{array}\right.
\end{equation}
where\begin{equation*}
\begin{aligned}
&\D \sigma^\alpha= \alpha^2\p_t\,\Delta u^\alpha+\alpha^2(u^\alpha\cdot\nabla)\Delta u^\alpha+\alpha^2\sum_{j=1}^2(\Delta u^\alpha_j)\nabla u_j^\alpha.
\end{aligned}
\end{equation*}
Multiply \eqref{28} by $W^\alpha$ and integrate on $\Omega\times [0,t]$. After integrating by parts,  we obtain
\begin{equation} \label{Wid}
\begin{aligned}
& \frac{1}{2}\|W^\alpha(t)\|^2 = \frac{1}{2}\|W^\alpha(0)\|^2 - \int_0^t\int_\Omega [(W^\alpha\cdot\nabla)\bar{u}]\cdot W^\alpha\dif x\dif s \\
& + \int_0^t\int_{\Omega}\di \sigma^\alpha\cdot W^\alpha\dif x\dif s,~\mbox{for all}~t\in[0,T].
\end{aligned}
\end{equation}
Clearly the second term on the right-hand side may be estimated by
\begin{equation} \label{secondtermRHS}
\begin{aligned}
\left|\int_0^t\int_\Omega [(W^\alpha\cdot\nabla)\bar{u}]\cdot W^\alpha\dif x\dif s \right| & \le\|\nabla\bar{u}\|_{L^\infty(\Omega\times (0,T))}\int_0^t\|W^\alpha(s)\|^2
\dif s \\
& \leq K\int_0^t\|W^\alpha(s)\|^2 \dif s
\end{aligned}
\end{equation}
We also  have, for every $t\in [0,T]$,
\begin{equation}\label{29}
\begin{aligned}
\int_0^t\int_{\Omega}\D \sigma^\alpha\cdot W^\alpha\dif x\dif s&=\alpha^2\int_0^t\int_\Omega\p_s\Delta u^\alpha\cdot W^\alpha\dif x\dif s\\
&+\alpha^2\int_0^t\int_\Omega[(u^\alpha\cdot\nabla)\Delta u^\alpha]\cdot W^\alpha\dif x\dif s\\
&+\alpha^2\int_0^t\int_\Omega\sum_{j=1}^2(\Delta u^\alpha_j)\nabla u_j^\alpha\cdot W^\alpha\dif x\dif s\\
&=:I_1(t)+I_2(t)+I_3(t).
\end{aligned}
\end{equation}

We will examine each of the terms in \eqref{29}.
We begin by estimating  $I_1(t)$. Notice that the main difficulty arises from the fact that only $\bar{u}\cdot\hat{n}=0$ on $\p\o$,  while the vector field $\bar{u}$ might not vanish on $\p\o$. However, the basic step, as we will see below, in Kato's argument is to consider instead $(\bar{u}-u_b)$.  Therefore, we have:

\begin{equation*}
\begin{aligned}
&I_1(t) =  \alpha^2 \int_0^t\int_\Omega \p_s\Delta u^\alpha\cdot W^\alpha\dif x\dif s=
\alpha^2 \int_0^t \int_{\Omega} \partial_s \Delta u^{\alpha} \cdot u^{\alpha} \dif x \dif s -
\alpha^2 \int_0^t \int_{\Omega} \partial_s \Delta u^{\alpha} \cdot \bar{u} \dif x \dif s\\
&=- \alpha^2 \int_0^t\int_{\Omega} \partial_s \nabla u^\alpha \,\cdot \,\nabla u^\alpha \dif x \dif s
 - \alpha^2 \int_0^t \int_{\Omega} \partial_s \Delta u^{\alpha} \cdot (\bar{u} - u_b) \dif x \dif s- \alpha^2 \int_0^t \int_{\Omega} \partial_s \Delta u^{\alpha} \cdot u_b \dif x \dif s\\
& =-\frac{\alpha^2}{2} \|\nabla u^\alpha (t)\|^2 + \frac{\alpha^2}{2} \|\nabla u^\alpha_0 \|^2+ \alpha^2 \int_0^t \int_{\Omega} \partial_s \nabla u^{\alpha} \,\cdot \,\nabla (\bar{u} - u_b) \dif x \dif s
- \alpha^2 \int_0^t \int_{\Omega} \partial_s \Delta u^{\alpha} \cdot u_b \dif x \dif s \\
&=-\frac{\alpha^2}{2} \|\nabla u^\alpha (t)\|^2 + \frac{\alpha^2}{2} \|\nabla u^\alpha_0 \|^2  - \alpha^2 \int_0^t \int_{\Omega} \nabla u^{\alpha} \,\cdot \,\partial_s \nabla (\bar{u} - u_b) \dif x \dif s\\
&-\alpha^2 \int_{\Omega} \nabla u^{\alpha}_0 \cdot (\nabla \bar{u}_0 - \nabla u_b(0) ) \dif x
+\alpha^2 \int_{\Omega} \nabla u^{\alpha}(t) \cdot (\nabla \bar{u}(t) - \nabla u_b(t) ) \dif x
\\
&+ \alpha^2 \int_0^t \int_{\Omega} \Delta u^{\alpha} \, \cdot \, \partial_s u_b
 \dif x \dif s + \alpha^2 \int_{\Omega} \Delta u^\alpha_0 \cdot u_b(0) \dif x- \alpha^2 \int_{\Omega} \Delta u^\alpha (t) \cdot u_b(t) \dif x .
\end{aligned}
\end{equation*}

With this identity we can estimate $I_1(t)$, for all $t\in[0,T]$,
\begin{align*}
&I_1(t) \leq -\frac{\alpha^2}{2} \|\nabla u^{\alpha}(t)\|^2 +  \frac{\alpha^2}{2} \|\nabla u^{\alpha}_0\|^2- \alpha^2\int_{\Omega} \nabla u_0^{\alpha}\cdot \nabla \bar{u}_0 \dif x + \alpha^2 \|\nabla u_0^\alpha\| \|\nabla u_b(0)\| \\
&+ \alpha^2 \|\nabla u^\alpha(t) \| \|\nabla \bar{u} (t)\| + \alpha^2 \|\nabla u^\alpha (t) \| \|\nabla u_b(t)\|+\alpha^2 \|\Delta u^{\alpha}_0\|\|u_b(0)\| + \alpha^2 \|\Delta u^{\alpha}(t)\|\|u_b(t)\| \\
&+\alpha^2 \int_0^t \|\nabla u^{\alpha} \| \|\partial_s \nabla \bar{u}\| \dif s
+ \alpha^2\int_0^t \|\nabla u^{\alpha} \| \|\partial_s \nabla u_b \| \dif s
+ \alpha^2\int_0^t \| \Delta u^{\alpha} \| \|\partial_s u_b \| \dif s
\\
& \leq -\frac{\alpha^2}{2} \|\nabla u^{\alpha}(t)\|^2 +  \frac{\alpha^2}{2} \|\nabla u^{\alpha}_0\|^2
- \alpha^2\int_{\Omega} \nabla u_0^{\alpha}\cdot \nabla \bar{u}_0 \dif x +K \alpha^2 \delta^{-1/2} \|\nabla u_0^\alpha\| \\
&+ \frac{\alpha^2}{16} \|\nabla u^\alpha(t) \|^2 +  K \alpha^2\|\nabla \bar{u} (t)\|_{L^{\infty}((0,T);L^2)}^2 + \frac{\alpha^2}{16} \|\nabla u^\alpha(t) \|^2 +  K \alpha^2\delta^{-1}\\
&+\alpha^2 \|\Delta u^{\alpha}_0\|\|u_b(0)\| + \alpha^2 \|\Delta u^{\alpha}(t)\|\|u_b(t)\| \\
&+\alpha^2 \int_0^t \|\nabla u^{\alpha} \| \|\partial_s \nabla \bar{u}\| \dif s
+ \alpha^2\int_0^t \|\nabla u^{\alpha} \| \|\partial_s \nabla u_b \| \dif s
+ \alpha^2\int_0^t \| \Delta u^{\alpha} \| \|\partial_s u_b \| \dif s,
\end{align*}
where we have used above Young's inequality together with the estimates from Lemma \ref{ubest}.

Next, we recall the following inequality for functions in $H^3$:
\begin{equation}\label{interinq}
 \|\Delta f \| \leq K \|\nabla f \|^{1/2}\|f\|_{H^3}^{1/2}.
 \end{equation}

Let us continue to bound $I_1$. We use \eqref{interinq} and the fact that $\bar{u}\in C^1([0,T]; H^2(\o))$ to obtain,  for all $t\in [0,T]$
\begin{align*}
I_1(t)  &\leq -\frac{\alpha^2}{2} \|\nabla u^{\alpha}(t)\|^2 +  \frac{\alpha^2}{2} \|\nabla u^{\alpha}_0\|^2- \alpha^2\int_{\Omega} \nabla u_0^{\alpha}\cdot \nabla \bar{u}_0 \dif x +K \alpha^2 \delta^{-1/2} \|\nabla u_0^\alpha\| \\
&+ \frac{\alpha^2}{16} \|\nabla u^\alpha(t) \|^2 +  K \alpha^2\|\nabla \bar{u} (t)\|_{L^{\infty}((0,T);L^2)}^2 + \frac{\alpha^2}{16} \|\nabla u^\alpha(t) \|^2 +  K \alpha^2\delta^{-1}\\
&+K\alpha^2 \|\nabla u^{\alpha}_0\|^{1/2} \|u_0^\alpha\|_{H^3}^{1/2} \|u_b(0)\| +K \alpha^2 \|\nabla u^{\alpha}(t)\|^{1/2}\|u^\alpha(t) \|_{H^3}^{1/2}\|u_b(t)\|, \\
&+\alpha^2 \int_0^t \|\nabla u^{\alpha} \|^2 \dif s + \alpha^2\int_0^t  \|\partial_s \nabla \bar{u}\|^2 \dif s+ \alpha^2\int_0^t \|\nabla u^{\alpha} \|^2 \dif s \\
 &+ \alpha^2\int_0^t  \|\partial_s \nabla u_b \|^2\dif s+ K\alpha^2\int_0^t \| \nabla u^{\alpha} \|^{1/2} \| u^{\alpha} \|_{H^3}^{1/2}  \|\partial_s u_b \| \dif s.
\end{align*}
Using \eqref{H1estualpha} and \eqref{281} together with estimates from Lemma \ref{ubest}, we obtain
\begin{align*}
I_1(t)&\leq -\frac{\alpha^2}{2} \|\nabla u^{\alpha}(t)\|^2 +  \frac{\alpha^2}{2} \|\nabla u^{\alpha}_0\|^2- \alpha^2\int_{\Omega} \nabla u_0^{\alpha}\cdot \nabla \bar{u}_0 \dif x +K \alpha^2 \delta^{-1/2} \frac{1}{\alpha} \\
&+ \frac{\alpha^2}{16} \|\nabla u^\alpha(t) \|^2 +  K \alpha^2\|\nabla \bar{u} (t)\|_{L^{\infty}((0,T);L^2)}^2 + \frac{\alpha^2}{16} \|\nabla u^\alpha(t) \|^2 +  K \alpha^2\delta^{-1}\\
&+K\alpha^2 \left(\frac{1}{\alpha} \right)^{1/2} \left(\frac{1}{\alpha^3}\right)^{1/2} \delta^{1/2} + \frac{\alpha^2}{8} \|\nabla u^{\alpha}(t)\|^2 + K\alpha^2 \|u^\alpha(t) \|_{H^3}^{2/3}\|u_b(t)\|^{4/3} \\
&+\alpha^2 \int_0^t \|\nabla u^{\alpha} \|^2 \dif s + \alpha^2\int_0^t  \|\partial_s \nabla \bar{u}\|^2 \dif s+ \alpha^2\int_0^t \|\nabla u^{\alpha} \|^2 \dif s +  \alpha^2\int_0^t  \|\partial_s \nabla u_b \|^2\dif s\\
&+ \alpha^2\int_0^t \| \nabla u^{\alpha} \|^2 \dif s + K\alpha^2 \int_0^t \| u^{\alpha} \|_{H^3}^{2/3}  \|\partial_s u_b \|^{4/2} \dif s.
\end{align*}
Therefore, coalescing similar terms we obtain, for all $t\in [0,T]$,
\begin{align*}
I_1(t) &\leq -\frac{\alpha^2}{4} \|\nabla u^{\alpha}(t)\|^2 +  \frac{\alpha^2}{2} \|\nabla u^{\alpha}_0\|^2- \alpha^2\int_{\Omega} \nabla u_0^{\alpha}\cdot \nabla \bar{u}_0 \dif x + K \alpha \delta^{-1/2} \\
&+ K\alpha^2 + K \alpha^2 \delta^{-1} +  K \delta^{1/2} + K\alpha^2 \left( \frac{1}{\alpha^3} \right)^{2/3}(\delta^{1/2})^{4/3}\\
&+ K \alpha^2\int_0^t \|\nabla u^{\alpha}\|^2 \dif s + K\alpha^2 T +  K \alpha^2 T \delta^{-1} + K \alpha^2 T \left( \frac{1}{\alpha^3} \right)^{2/3}(\delta^{1/2})^{4/3}.
\end{align*}

Thus, for all $t\in [0,T],$ we have
\begin{equation} \label{I1est}
\begin{aligned}
&I_1(t) \leq -\frac{\alpha^2}{4} \|\nabla u^{\alpha}(t)\|^2 +  \frac{\alpha^2}{2} \|\nabla u^{\alpha}_0\|^2 +K \alpha^2\int_0^t \|\nabla u^{\alpha}\|^2 \dif s
\\
&- \alpha^2\int_{\Omega} \nabla u_0^{\alpha}\cdot \nabla \bar{u}_0 \dif x +K \alpha\delta^{-1/2}  + K\alpha^2 + K \alpha^2\delta^{-1}+K\delta^{1/2}  + K\delta^{2/3}\\
& = -\frac{\alpha^2}{4} \|\nabla u^{\alpha}(t)\|^2 +\frac{\alpha^2}{2} \|\nabla u^{\alpha}_0\|^2+ K \alpha^2\int_0^t \|\nabla u^{\alpha}\|^2 \dif s  + g(\alpha, u_0^{\alpha},\bar{u}_0),
\end{aligned}
\end{equation}
with
\begin{align*}
g(\alpha,u_0^\alpha,\bar{u}_0) &=  - \alpha^2\int_{\Omega} \nabla u_0^{\alpha}\cdot \nabla \bar{u}_0 \dif x+ K\alpha^2 +
 K \alpha\delta^{-1/2}   + K \alpha^2\delta^{-1}+K\delta^{1/2}  + K\delta^{2/3}.
 \end{align*}
Now, we choose $\delta=\delta(\alpha)$ such that
\begin{equation}\label{deltadata}
\delta(\alpha)\to 0~\mbox{and}~\frac{\alpha^2}{\delta(\alpha)}\to 0,~\mbox{as}~\alpha\to 0.
\end{equation}
Therefore, it follows from  the assumption \eqref{deltadata} and the hypotheses of  Theorem \ref{mainth} that
\begin{equation}\label{E2}
g(\alpha,u_0^\alpha,\bar{u}_0) \to 0, \;\;\text{ as } \;\;\alpha \to 0.
\end{equation}

Next, we examine $I_2$ and $I_3$. We start by noticing, after integrating by parts, that, for all $t\in [0,T],$
\begin{align*}
I_2(t) + I_3(t) &:= \alpha^2\int_0^t\int_\Omega[(u^\alpha\cdot\nabla)\Delta u^\alpha]\cdot W^\alpha\dif x\dif s +\alpha^2\int_0^t\int_\Omega\sum_{j=1}^2(\Delta u^\alpha_j)\nabla u_j^\alpha\cdot W^\alpha\dif x\dif s \\
&= \alpha^2\int_0^t\int_\Omega[(u^\alpha\cdot\nabla)\Delta u^\alpha]\cdot u^\alpha \dif x\dif s
- \alpha^2\int_0^t\int_\Omega[(u^\alpha\cdot\nabla)\Delta u^\alpha]\cdot \bar{u} \dif x\dif s\\
&+\alpha^2\int_0^t\int_\Omega\sum_{j=1}^2(\Delta u^\alpha_j)\nabla u_j^\alpha\cdot u^\alpha \dif x\dif s
- \alpha^2\int_0^t\int_\Omega\sum_{j=1}^2(\Delta u^\alpha_j)\nabla u_j^\alpha\cdot \bar{u} \dif x\dif s\\
&=\alpha^2\int_0^t\int_\Omega[(u^\alpha\cdot\nabla)\Delta u^\alpha]\cdot u^\alpha \dif x\dif s
+\alpha^2\int_0^t\int_\Omega \Delta u^\alpha \cdot [(u^\alpha \cdot \nabla) u^\alpha] \dif x\dif s\\
&- \alpha^2\int_0^t\int_\Omega[(u^\alpha\cdot\nabla)\Delta u^\alpha]\cdot \bar{u} \dif x\dif s
- \alpha^2\int_0^t\int_\Omega\sum_{j=1}^2(\Delta u^\alpha_j)\nabla u_j^\alpha\cdot \bar{u} \dif x\dif s.
\end{align*}
Notice that since $\D u^\alpha=0$ and $u^\alpha$ vanishes on  $\p\Omega$, we can integrate by parts to show that
$$
\alpha^2\int_0^t\int_\Omega[(u^\alpha\cdot\nabla)\Delta u^\alpha]\cdot u^\alpha \dif x\dif s
+\alpha^2\int_0^t\int_\Omega \Delta u^\alpha \cdot [(u^\alpha \cdot \nabla) u^\alpha] \dif x\dif s =0.
$$
As a result of all the above we have
\begin{align*}
I_2+I_3&=- \alpha^2\int_0^t\int_\Omega[(u^\alpha\cdot\nabla)\Delta u^\alpha]\cdot \bar{u} \dif x\dif s
- \alpha^2\int_0^t\int_\Omega\sum_{j=1}^2(\Delta u^\alpha_j)\nabla u_j^\alpha\cdot \bar{u} \dif x\dif s\\
&=: I_2'(t) + I_3'(t).
\end{align*}

We now estimate $I_2'(t)$, for all $t\in [0,T]$,
\begin{align*}
&I_2'(t) = - \alpha^2\int_0^t\int_\Omega[(u^\alpha\cdot\nabla)\Delta u^\alpha]\cdot \bar{u} \dif x\dif s = \alpha^2\int_0^t\int_\Omega \Delta u^\alpha\cdot [(u^\alpha\cdot\nabla) \bar{u}] \dif x\dif s\\
&= -\alpha^2\int_0^t\int_\Omega \sum_{k=1}^2 \partial_k u^\alpha \cdot \partial_k [(u^\alpha\cdot\nabla) \bar{u}] \dif x\dif s\\
&= -\alpha^2\int_0^t\int_\Omega \sum_{k=1}^2 \partial_k u^\alpha \cdot [(\partial_k u^\alpha\cdot\nabla) \bar{u}] \dif x\dif s-\alpha^2\int_0^t\int_\Omega \sum_{k=1}^2 \partial_k u^\alpha \cdot [( u^\alpha\cdot\nabla) \partial_k\bar{u} ]\dif x\dif s.
\end{align*}
Using the fact that $\bar{u}\in C([0,T];(H^3(\o))^2)\cap C^1([0,T]; (H^2(\o))^2)$, we obtain, for all $t\in [0,T],$
\begin{equation*}
\begin{aligned}
I'_2(t)&\leq \alpha^2 \|\nabla \bar{u}\|_{L^\infty((0,T)\times\Omega)} \int_0^t \|\nabla u^\alpha(s)\|^2\dif s +
\alpha^2\int_0^t \|\nabla u^{\alpha}(s) \|\|u^\alpha(s)\|_{L^4} \|D^2 \bar{u}(s)\|_{L^4} \dif s
\\
&\leq K\alpha^2 \|\bar{u}\|_{L^\infty((0,T);H^3)} \int_0^t \|\nabla u^\alpha\|^2 \dif s +
K\alpha^2\int_0^t \|u^\alpha\|^{1/2} \|\nabla u^\alpha \|^{3/2} \|D^2 \bar{u}\|_{L^4} \dif s,\\
&\leq K\alpha^2 \|\bar{u}\|_{L^\infty((0,T);H^3)} \int_0^t \|\nabla u^\alpha\|^2 \dif s +
K \alpha^2\|\bar{u}\|_{L^\infty((0,T);H^3)}^4 \int_0^t \|u^\alpha\|^2 \dif s\\
&+ K\alpha^2 \int_0^t \|\nabla u^\alpha \|^2 \dif s,
\end{aligned}
\end{equation*}
where we used  the 2D-Ladyzhenskaya inequality followed by Young's inequality in the last bound.
Hence we find, after piecing together similar terms that, for every $t\in [0,T]$, we have

\begin{equation} \label{I2primeest}
I_2'(t) \leq K \alpha^2  \int_0^t \|\nabla u^\alpha\|^2 \dif s +
K \alpha^2 T.
\end{equation}

Finally, we turn to $I_3'$. Here again we will not be able to  integrate by parts, since  only $\bar{u}\cdot\hat{n}=0$ on $\p\o$, while the vector field $\bar{u}$ might not vanish on $\p\o$. To remedy this situation we consider instead the vector field $\bar{u}-u_b,$ where we have explicit understanding, thanks to Lemma \ref{ubest} of the behavior of $u_b$ at $\p\o.$ Thus we have
\begin{equation*}
\begin{aligned}
I_3' (t)&= - \alpha^2\int_0^t\int_\Omega\sum_{j=1}^2(\Delta u^\alpha_j)\nabla u_j^\alpha\cdot \bar{u} \dif x\dif s\\
&= - \alpha^2\int_0^t\int_\Omega\sum_{j=1}^2(\Delta u^\alpha_j)\nabla u_j^\alpha\cdot (\bar{u} - u_b) \dif x\dif s
 - \alpha^2\int_0^t\int_\Omega\sum_{j=1}^2(\Delta u^\alpha_j)\nabla u_j^\alpha\cdot u_b \dif x\dif s \\
&=:J_1(t) + J_2(t).
\end{aligned}
\end{equation*}

Note that, for every $t\in[0,T],$ we have
\begin{equation*}
\begin{aligned}
&J_1(t) := - \alpha^2\int_0^t\int_\Omega\sum_{j=1}^2(\Delta u^\alpha_j)\nabla u_j^\alpha\cdot (\bar{u} - u_b) \dif x\dif s\\
&=  \alpha^2\int_0^t\int_\Omega\sum_{j,k=1}^2(\partial_k u^\alpha_j)\nabla u_j^\alpha\cdot \partial_k(\bar{u} - u_b) \dif x\dif s
+ \alpha^2\int_0^t\int_\Omega\sum_{j,k=1}^2(\partial_k u^\alpha_j)\partial_k \nabla u_j^\alpha\cdot (\bar{u} - u_b) \dif x\dif s
\\
&=  \alpha^2\int_0^t\int_\Omega\sum_{j,k=1}^2(\partial_k u^\alpha_j)\nabla u_j^\alpha\cdot \partial_k(\bar{u} - u_b) \dif x\dif s
+ \alpha^2\int_0^t\int_\Omega\sum_{j,k=1}^2(\bar{u} - u_b)\cdot \nabla \left[\frac{|\partial_k u_j^\alpha|^2}{2} \right]\dif x\dif s
\\
&=  \alpha^2\int_0^t\int_\Omega\sum_{j,k=1}^2(\partial_k u^\alpha_j)\nabla u_j^\alpha\cdot \partial_k \bar{u}  \dif x\dif s  -\alpha^2\int_0^t\int_\Omega\sum_{j,k=1}^2(\partial_k u^\alpha_j)\nabla u_j^\alpha\cdot \partial_k u_b \dif x\dif s\\
&\leq \alpha^2 \|\nabla \bar{u}\|_{L^\infty((0,T)\times\Omega)}\int_0^t \|\nabla u^\alpha \|^2 \dif s
-\alpha^2\int_0^t\int_\Omega\sum_{\ell,j,k=1}^2(\partial_k u^\alpha_j)\,(\partial_{\ell} u_j^\alpha)\cdot \partial_k (u_b)_{\ell} \dif x\dif s.
\end{aligned}
\end{equation*}
Therefore, after integrating by parts we obtain
\begin{equation*}
\begin{aligned}
J_1(t) &\le K \alpha^2\|\bar{u}\|_{L^\infty((0,T);H^3)}\int_0^t \|\nabla u^\alpha \|^2 \dif s
+\alpha^2\int_0^t\int_\Omega\sum_{\ell,j,k=1}^2(\partial_{\ell}\partial_k u^\alpha_j) \,u_j^\alpha \partial_k (u_b)_{\ell} \dif x\dif s\\
&= K\alpha^2\|\bar{u}\|_{L^\infty((0,T);(H^3(\o))^2)}\int_0^t \|\nabla u^\alpha \|^2 \dif s\\
&-\alpha^2\int_0^t\int_\Omega\sum_{\ell,j,k=1}^2(\partial_{\ell}\partial_k \partial_k u^\alpha_j) \,u_j^\alpha (u_b)_{\ell} \dif x\dif s
-\alpha^2\int_0^t\int_\Omega\sum_{\ell,j,k=1}^2(\partial_{\ell} \partial_k u^\alpha_j) \,\partial_k u_j^\alpha (u_b)_{\ell} \dif x\dif s\\
&= K\alpha^2\|\bar{u}\|_{L^\infty((0,T);H^3)}\int_0^t \|\nabla u^\alpha \|^2 \dif s\\
&-\alpha^2\int_0^t\int_\Omega\sum_{\ell,j,k=1}^2(\partial_{\ell}\partial_k \partial_k u^\alpha_j) \,u_j^\alpha (u_b)_{\ell} \dif x\dif s
-\alpha^2\int_0^t\int_\Omega \sum_{j,k=1}^2 (u_b \cdot \nabla) \left[\frac{|\partial_k u^\alpha_j|^2}{2}\right] \dif x\dif s\\
&= K\alpha^2\|\bar{u}\|_{L^\infty((0,T);H^3)}\int_0^t \|\nabla u^\alpha \|^2 \dif s-\alpha^2\int_0^t\int_\Omega\sum_{\ell,j,k=1}^2(\partial_{\ell}\partial_k \partial_k u^\alpha_j) \,u_j^\alpha\cdot \,(u_b)_{\ell} \dif x\dif s
\end{aligned}
\end{equation*}
\begin{equation*}
\begin{aligned}
&=K \alpha^2\|\bar{u}\|_{L^\infty((0,T);H^3)}\int_0^t \|\nabla u^\alpha \|^2 \dif s+\alpha^2\int_0^t\int_\Omega\sum_{\ell,j,k=1}^2(\partial_k \partial_k u^\alpha_j) \,\partial_{\ell} u_j^\alpha \,(u_b)_{\ell} \dif x\dif s
\\
&=K \alpha^2\|\bar{u}\|_{L^\infty((0,T);H^3)}\int_0^t \|\nabla u^\alpha(s) \|^2 \dif s
+\alpha^2\int_0^t\int_\Omega (\Delta u^\alpha_j )\,\nabla u_j^\alpha\cdot u_b \dif x\dif s.
\end{aligned}
\end{equation*}
Consequently,
\[J_1(t) \leq  K \alpha^2 \int_0^t \|\nabla u^\alpha(s) \|^2 \dif s -J_2(t),
\]
for all $t\in [0,T]$.
As a result, we have obtained that
\begin{equation} \label{I3primeest}
I_3'(t) = J_1(t) + J_2(t) \leq K \alpha^2 \int_0^t \|\nabla u^\alpha(s) \|^2 \dif s,
\end{equation}
for all $t\in [0,T]$.
Recalling \eqref{29} and putting together the estimates in \eqref{I1est}, \eqref{I2primeest} and \eqref{I3primeest} we deduce that
\begin{equation} \label{RHSest}
\begin{aligned}
&\int_0^t\int_{\Omega}\D \sigma^\alpha\cdot W^\alpha\dif x\dif s =I_1(t)+I_2(t)+I_3(t) \\
&\leq -\frac{\alpha^2}{4} \|\nabla u^{\alpha}(t)\|^2 + K \alpha^2\int_0^t \|\nabla u^{\alpha}\|^2 \dif s  + g(\alpha, u_0^{\alpha},\bar{u}_0)\\
& + K \alpha^2  \int_0^t \|\nabla u^\alpha\|^2 \dif s +
K \alpha^2 T + K \alpha^2 \int_0^t \|\nabla u^\alpha \|^2 \dif s.
\end{aligned}
\end{equation}

We insert \eqref{secondtermRHS} and \eqref{RHSest} into \eqref{Wid} to conclude

\begin{equation} \label{West}
\begin{aligned}
& \frac{1}{2}\|W^\alpha(t)\|^2 \leq  \frac{1}{2}\|W^\alpha(0)\|^2 + K\int_0^t\|W^\alpha\|^2 \dif s \\
& -\frac{\alpha^2}{4} \|\nabla u^{\alpha}(t)\|^2 + K \alpha^2\int_0^t \|\nabla u^{\alpha}\|^2 \dif s + \frac{\alpha^2}{2} \|\nabla u^{\alpha}_0\|^2 +  g(\alpha, u_0^{\alpha},\bar{u}_0) \\
&+ K \alpha^2  \int_0^t \|\nabla u^\alpha\|^2 \dif s +
K \alpha^2 T + K \alpha^2 \int_0^t \|\nabla u^\alpha \|^2 \dif s.
\end{aligned}
\end{equation}
We can rewrite \eqref{West} as
\begin{equation} \label{West2}
\begin{aligned}
& \|W^\alpha(t)\|^2 + \alpha^2\|\nabla u^{\alpha}(t)\|^2 \leq  K_1(\|W^\alpha(0)\|^2 + \alpha^2 \|\nabla u^{\alpha}_0\|^2)\\
& + K_2\int_0^t (\|W^\alpha\|^2 + \alpha^2 \|\nabla u^{\alpha}\|^2 )\dif s +  \tilde{g}(\alpha, u_0^{\alpha},\bar{u}_0),
\end{aligned}
\end{equation}
where
\begin{equation}\label{tg}
\tilde{g}(\alpha, u_0^{\alpha},\bar{u}_0)=g(\alpha, u_0^{\alpha},\bar{u}_0)+
KT \alpha^2.\end{equation}
Applying Gronwall's lemma to \eqref{West2},  we obtain
 \begin{equation*}
 \begin{aligned}
 \sup_{t\in[0,T]}\left(\|W^\alpha(t)\|^2+\alpha^2\|\nabla u^\alpha(t)\|^2\right)\le e^{K_2T}\left[K_1(\|W^\alpha(0)\|^2+\alpha^2\|\nabla u^\alpha_0\|^2)+\tilde{g}(\alpha, u_0^{\alpha},\bar{u}_0)\right],
 \end{aligned}
 \end{equation*}
 where $K_1,K_2$ do not depend on $\alpha$.

 Thanks to \eqref{E1}, \eqref{E2} and \eqref{tg},  we conclude that
 $$\sup_{t\in (0,T)}(\|u^\alpha(t)-\bar{u}(t)\|^2+\alpha^2\|\nabla u^\alpha(t)\|^2)\to 0,$$ as $\alpha\to 0$.

 \end{proof}

\section{Comments and conclusions}

In our main result, Theorem \ref{mainth}, we assume that the initial data for the Euler equations belongs to $(H^3(\Omega))^2$, is divergence free and satisfies $u_0 \cdot \hat{n}=0$. In addition, we postulate the existence of a suitable family of approximations to $u_0$, i.e. a family of approximations verifying \eqref{E1}, $\{u_0^{\alpha}\} \subset (H^3(\Omega))^2$. A natural question which arises is whether such approximations exist for any, given, $u_0$ as above. We begin this section by providing a construction of such an approximation. In fact, in the following result, concerning the construction of $u^\alpha_0$, we require considerably less regularity from $u_0$.

\begin{Proposition}
Let $u_0\in H\cap (H^1(\Omega))^2$. Then there exists a suitable family of approximations to $u_0$, $\{u_0^\alpha\}$ satisfying \eqref{E1}.
\end{Proposition}

\begin{proof}
  Let us denote by $P_\sigma$ the Leray-Helmholtz projector operator, i.e.  the orthogonal projection from $(L^2(\Omega))^2$ onto $H$.  We denote by $A= P_{\sigma}(-\Delta)$ the Stokes operator, with $D(A)=(H^2(\o))^2\cap V$. It is well known that the space $H$ possesses an orthonormal  basis $\{w_j\}_{i=1}^\infty$ of eigenfunctions of $A$, with corresponding eigenvalues $\lambda_j, j=1,2,\cdots,$ i.e. $Aw_j=\lambda_j w_j$ (cf.\cite{cf}). Moreover, it is well known that $\lambda_j \thicksim j\lambda_1$, for $j=1,2,\cdots,$ see, e.g., \cite{I, Me}. Let us set $H_m= \mbox{span} \{w_1,w_2,\cdots, w_m\}$ and  by     $P_m$ to be the orthogonal projection from $H$ onto $H_m$.

 Let $u_0\in H^1(\o)\cap H$, we set
 \[u_0^\alpha=P_m u_0=\sum_{i=1}^m(u_0,w_j)w_j,\]
 where we choose $m= \lfloor\frac{1}{\alpha^2\lambda_1}\rfloor$.

It is clear that $\|u^\alpha_0-u_0\|\to 0$, as $\alpha\to 0$, and that
$u^{\alpha}_0 = 0$ on $\partial \Omega$. Therefore, conditions $(i)$ and $(ii)$ of \eqref{E1} are met.

We observe that for every $s\ge 0$, there exists a constant $K>0$, which depends on $s$, but is  independent of $\alpha$, so that
\begin{equation}\label{hnorm}
\|u^\alpha_0\|^2_{H^s}\le K \sum_{j=1}^m\lambda_j^s|(u_0,w_j)|^2\le K\lambda_m^s\|u_0\|^2\le K\alpha^{-2s}\|u_0\|^2.
\end{equation}
Setting $s=3$ in \eqref{hnorm} implies condition  $(iv)$ of \eqref{E1}.

All that remains  to verify is condition $(iii)$ of \eqref{E1}.  Observe that
\begin{equation}\label{ithird}
\begin{aligned}
\|\nabla u^\alpha_0\|^2&=\|\nabla P_m u_0\|^2=\|{A}^\frac{1}{2}  P_m u_0\|^2=(P_m u_0,{ A} P_m u_0)\\[2mm]
&=(u_0, P_m {  A} P_m u_0)=(u_0, { A} P_m u_0)=(u_0, P_\sigma(-\Delta) P_m u_0)\\
&=(u_0, (-\Delta) P_m u_0)=\int_\o (\nabla u_0:\nabla u^\alpha_0)\dif x-\int_{\p\o}u_0\cdot\frac{\p u^\alpha_0}{\p\hat{n}}\dif\Gamma\\
&\le\|\nabla u_0\|\|\nabla u^\alpha_0\| + \|u_0\|_{L^2(\partial \o)}\|\nabla u^\alpha_0\|_{L^2(\partial \o)}\\
&\le\|\nabla u_0\|\|\nabla u^\alpha_0\|+K\|u_0\|_1\|\nabla u^\alpha_0\|^\frac{1}{2}\|u^\alpha_0\|_2^\frac{1}{2},
\end{aligned}
\end{equation}
where the last inequality is obtained by using the following boundary trace inequality \cite{ga2}
$$\|f\|^2_{L^2(\p\o)}\le K\|f\|\|  f\|_1.$$
By virtue of Young's inequality, \eqref{ithird} implies
\begin{equation}\label{ythird}
\begin{aligned}
\alpha^2\|\nabla u^\alpha_0\|^2\le K\alpha^2\|\nabla u_0\|^2+K\alpha^2\|u^\alpha_0\|_2^\frac{2}{3}\|u_0\|_1^\frac{4}{3}.
\end{aligned}
\end{equation}
Using \eqref{hnorm}, for $s=2,$ we find that
$$\alpha^2\|u^\alpha_0\|_2^\frac{2}{3}\|u_0\|_1^\frac{4}{3}\le K\alpha^\frac{2}{3}\|u_0\|^\frac{2}{3}\|u_0\|_1^\frac{4}{3}\le K\alpha^\frac{2}{3}\|u_0\|_1^2.$$
Thus it follows from the above and \eqref{ythird} that
$$\alpha^2\|\nabla u^\alpha_0\|^2\le K(\alpha^2+\alpha^\frac{2}{3})\| u_0\|_1^2.$$
Hence, we obtain $(iii)$ of \eqref{E1} as desired.

\end{proof}

Our final result is an illustration of what we are calling {\it boundary layer indifference} of the $\alpha \to 0$ limit. We consider $\Omega$
the infinite channel $\{0 < x_2 < 1, x_1 \in \mathbb{R}\}$, and we seek stationary  solutions of the Euler-$\alpha$ system of the form $u(x_1,x_2) = (\varphi(x_2),0)$,
known as {\it parallel flows}.

For the sake of comparison, let us first consider the Navier-Stokes equations, with viscosity $\nu > 0$ in a channel, with no-slip boundary conditions.
If we seek (stationary) parallel flow solutions for the Navier-Stokes equations in the context above, it is well-known that $\varphi$ must be the Poiseuille parabolic profile, which, for any viscosity $\nu > 0$, is given by $\varphi(x_2) = c x_2 (1 - x_2)$, for an arbitrary constant $c $. On the other hand, any parallel flow is
a stationary solution of the Euler equations in the channel,  and it is natural to ask which parallel flows are vanishing viscosity limits of stationary viscous flows. In fact,
if one considers $\nu$-dependent families of Poiseuille profiles, the only possible limits as $\nu \to 0$ are again of the form $c x_2 (1-x_2)$ (see the Prandtl-Batchelor Theorem, for example, in \cite{acheson} for a more thorough discussion of this issue).

The contrast of this rigid behavior with what happens with the Euler-$\alpha$ regularization is quite striking, as can be seen by the following result:

\begin{Proposition} Let $\varphi = \varphi(x_2)$ be any function in $C^2((0,1)) \cap C([0,1])$  with $\varphi(0) =\varphi(1)=0$. Then the velocity $u(x_1,x_2) = (\varphi(x_2),0)$  is a stationary solution of the Euler-$\alpha$ system for any $\alpha$, with pressure $$p = - \frac{\varphi^2 - \alpha^2 (\varphi^{\prime})^2}{2}.$$   \end{Proposition}

\begin{proof}
The two-dimensional stationary Euler-$\alpha$ system can be written in the form:
$$u \cdot \nabla ( u- \alpha^2 \Delta u) + \sum_j (u_j - \alpha^2 \Delta u_j) \nabla u_j = -\nabla p; \D u = 0$$
Setting $u = (\varphi(x_2),0)$, the divergence free condition is automatically satisfied, the horizontal momentum balance becomes
$-\partial_{x_1} p = 0$ and the vertical momentum balance equation becomes:
\[(\varphi - \alpha^2 \varphi^{\prime\prime})\varphi^{\prime}  = -\partial_{x_2} p, \]
so, taking $p$ as stated concludes the proof.
\end{proof}

As a consequence of this result, any parallel flow in the channel can be approximated in the $L^2-$norm, by stationary Euler-$\alpha$ solutions through the use of a cut-off function near the boundary of the channel, and by adjusting the pressure accordingly. The resulting boundary layer is of arbitrary width and profile.
This suggests that the hypothesis \eqref{E1} on the initial approximation could be a technical limitation of our proof, and not a sharp requirement.

There are many natural questions arising from the work we have presented. First, one may seek extensions to the three-dimensional case, the case of the second-grade fluid equation, and the case of other regularized models such as Leray-$\alpha$ and the Euler-Voigt-$\alpha$ models - a subject of a current research \cite{LNTZ}. Second, one may seek to optimize the regularity requirement on initial data, improve the space where convergence is taking place and find more precise estimates on error terms. Yet another avenue of investigation would be to examine the behavior of numerical approximations or implementations of $\alpha$-models with small $\alpha$ in domains with boundary. Finally, one may look for better understanding of the boundary layer, specially in time-dependent cases.\\

\noindent{\bf Acknowledgements}

E.S.T. would like to acknowledge the kind hospitality of the Universidade Federal do Rio de Janeiro (UFRJ) and Instituto Nacional de Matem\' {a}tica  Pura  e Aplicada (IMPA), where part of this work was completed.
The work of M.C.L.F. is partially supported by CNPq grant \# 303089 / 2010-5.
The work of H.J.N.L. is supported in part by CNPq grant \# 306331 / 2010-1 and FAPERJ grant \# E-26/103.197/2012.
The work of  E.S.T.  is supported in part by the NSF grants  DMS-1009950, DMS-1109640 and DMS-1109645, as well as by the Minerva Stiftung/Foundation. Also by CNPq-CsF grant \# 401615/2012-0, through the program Ci\^encia sem Fronteiras.
The work of A.B.Z. is supported in part by the CNPq-CsF grant \# 402694/2012-0, by the National Natural Science Foundation of China (11201411) and Jiangxi Provincial Natural Science Foundation of China (20122BAB211004), Higher Education  Teacher Training Foundation of Jiangxi Provincial Education Department  and Youth Innovation Group of Applied Mathematics  in Yichun University.

\end{document}